\documentclass[12 pt,a4paper,twoside,reqno]{amsart}
\usepackage{amsfonts,amssymb,amscd,amsmath,enumerate,verbatim,calc}

\newtheorem{thm}{Theorem}[section]

\newtheorem{cor}{Corollary}[section]

\newtheorem{defn}{Definition}[section]

\usepackage{amsmath, amssymb}
\usepackage{geometry}
\usepackage{hyperref}
\title{A Variant of the Fractional Gamma-Regularized Zeta Kernel}
\author{}
\date{today}

\begin{document}

\author{  Hunduma Legesse Geleta}

\date{\today}
\address{Department of Mathematics, Addis Ababa University, Addis Ababa, Ethiopia.}
\email{hunduma.legesse@aau.edu.et}

 \thanks{}


\begin{abstract}

In this paper, we introduce and study a special function, referred to as a variant of the fractional gamma-regularized zeta kernel, in which the lower incomplete gamma function appears in the numerator, contrasting with the fractional hypergeometric zeta function, where it appears in the denominator. We investigate its analytic properties, convergence behavior, and its connection to the Riemann zeta function. A representation is derived in terms of a multiplicative modification of the Riemann zeta function, and comparisons are made with known zero-free regions and deformation phenomena. We demonstrate that this function defines a deformation of the Riemann zeta function, which interpolates to the alternating zeta function in a specific limit. Additionally, a Dirichlet-type series representation is provided. Potential extensions, including a functional equation and asymptotic analysis, are proposed for future work.
\end{abstract}
\maketitle
\def\theequation{\thesection.\arabic{equation}}

\section*{Mathematics Subject Classification: 33xx, 11xx}

\section{Introduction and Motivation}
\setcounter{equation}{0}

The Riemann zeta function \( \zeta(s) \) is a cornerstone of analytic number theory \cite{TMA}. Numerous generalizations such as Hurwitz zeta function \cite{SR}, the hypergeometric zeta functions \cite{AHND}, Fractional hypergeometric zeta functions \cite{HA, HLG1}and deformations of \( \zeta(s) \), including those involving the incomplete gamma function \cite{HLG2, HA, HLG1}, have been developed in efforts to regularize divergent series, explore spectral analogues, and investigate profound properties such as the distribution of zeros \cite{AHN, Mu, DH, HLG2, LFA, GRSJ, SR}.\\

 This work may be viewed as a continuation of the studies presented in \cite{HA} and \cite{HLG2}. In \cite{HA}, Abdul Hassen and I introduced and analyzed a generalization of the classical Riemann zeta function via fractional hypergeometric zeta functions, using an integral representation. In \cite{HLG2} , I further examined the existence of zero-free regions for these fractional hypergeometric zeta functions in the right half-plane. In both works, gamma-regularized integral representations were employed, typically involving the incomplete gamma function in the denominator. This structure facilitated meaningful deformations of \( \zeta(s) \) under fractional parameters while preserving key analytic features.\\

Motivated by these developments, the present paper introduces a new variant of such regularizations in which the lower incomplete gamma function appears in the numerator. Our goal is to investigate the consequences of this structural inversion: particularly its effects on convergence, deformation behavior, the distribution of zeros, and the approximation of the classical \( \zeta(s) \) under fractional parameters and preserved key analytic properties. \\

This paper is organized as follows. In Section 2, we review key aspects of the Riemann zeta function and its generalization by fractional hypergeometric zeta function. In Sections 3 we present our main results. In particular in section 3.1 we formally define a variant of the fractional gamma-regularized zeta kernel and investigate its analytic properties and  Section 3.2 is devoted to the study of the zeros and zero-free regions of $k_a(s).$ Finally, in Section 4, we present our conclusions and suggest directions for future work.

\maketitle
\def\theequation{\thesection.\arabic{equation}}
\section{Preliminaries}

In this section, we review key aspects of the classical Riemann zeta function and its generalization via the fractional hypergeometric zeta function.

\begin{defn}
The \emph{Riemann zeta function} \( \zeta(s) \) is a function of a complex variable \( s = \sigma + it \), defined by the Dirichlet series
\[
\zeta(s) = \sum_{n=1}^{\infty} \frac{1}{n^s},
\]
which converges for all \( \sigma > 1 \).
\end{defn}

The Riemann zeta function also admits the integral representation
\[
\zeta(s) = \frac{1}{\Gamma(s)} \int_{0}^{\infty} \frac{x^{s-1}}{e^x - 1} \, dx,
\]
which provides a means for analytic continuation beyond the region of absolute convergence of the Dirichlet series.

An equivalent formulation is given by the Euler product:
\[
\zeta(s) = \prod_{p} \left(1 - p^{-s} \right)^{-1},
\]
where the product runs over all prime numbers \( p \). Since a convergent infinite product of non-zero terms cannot vanish, this representation implies that \( \zeta(s) \) is non-zero in the half-plane \( \sigma > 1 \), thereby establishing the zero-free region of the Riemann zeta function in that domain.

Another key identity is the celebrated functional equation:
\[
\zeta(s) = 2(2\pi)^{s-1} \zeta(1 - s) \Gamma(1 - s) \sin\left( \frac{\pi s}{2} \right),
\]
which plays a central role in the study of the distribution of zeros of \( \zeta(s) \).

\begin{defn}
Let \( a > 0 \) be a fixed real parameter. The \emph{fractional hypergeometric zeta function} of order \( a \), denoted \( \zeta_a(s) \), is defined by
\[
\zeta_a(s) = \frac{\Gamma(a)}{\Gamma(s + a - 1)} \int_0^{\infty} \frac{x^{s+a-2} e^{-x}}{\gamma(a, x)} \, dx,
\]
where \( \gamma(a,x) \) is the lower incomplete gamma function.
\end{defn}

This function is defined and analytic for \( \Re(s) = \sigma > 1 \) (see \cite{ref3}). Notably, when \( a = 1 \), the definition recovers the classical Riemann zeta function:
\[
\zeta_1(s) = \zeta(s).
\]

Using the substitution \( x = nu \), where \( n \in \mathbb{N} \), we obtain
\[
\zeta_a(s) = \frac{\Gamma(a)}{\Gamma(s + a - 1)} \int_0^{\infty} \frac{(nu)^{s+a-2} e^{-nu}}{\gamma(a, nu)} \, n \, du,
\]
which simplifies to
\[
\zeta_a(s) n^{1 - a - s} = \frac{\Gamma(a)}{\Gamma(s + a - 1)} \int_0^{\infty} \frac{u^{s+a-2} e^{-nu}}{\gamma(a, nu)} \, du.
\]

Summing both sides over all positive integers \( n \), we arrive at the identity
\[
\zeta_a(s) \zeta(s + a - 1) = \frac{\Gamma(a)}{\Gamma(s + a - 1)} \sum_{n=1}^{\infty} \int_0^{\infty} \frac{u^{s+a-2} e^{-nu}}{\gamma(a, nu)} \, du.
\]

This identity highlights the appearance of \( \gamma(a, nu) \) in the denominator of the integrand defining the fractional hypergeometric zeta function of order \( a \).

\maketitle
\def\theequation{\thesection.\arabic{equation}}
\section{Main Results}
In this section we present our results in two different subsections as follows. In the first subsection we introduce a variant of the fractional gamma-regularized zeta kernel, establish its domain of convergence, and derive a Dirichlet-type series representation. In the second subsection we investigate the analytic structure and zero set of the deformation function.
\subsection{Definition of the Variant Kernel Function}

In this section, we introduce a variant of the fractional gamma-regularized zeta kernel, establish its domain of convergence, and derive a Dirichlet-type series representation. This construction differs from the original formulation by placing the lower incomplete gamma function in the numerator, leading to a distinctive deformation of the classical Riemann zeta function.

\begin{defn}
Let \( a > 0 \). For \( \Re(s) > 1 \), we define the function \( k_a(s) \) by
\[
k_a(s) := \frac{1}{\Gamma(s)} \int_0^\infty x^{s-1} \sum_{n=1}^\infty \frac{e^{-n x} \gamma(a, n x)}{\Gamma(a)} \, dx,
\]
where \( \gamma(a, z) \) denotes the lower incomplete gamma function,
\[
\gamma(a, z) = \int_0^z t^{a-1} e^{-t} \, dt.
\]
\end{defn}

\paragraph{Remark.} Similar to the fractional hypergeometric zeta function (FHGZF), one may consider a shifted variant defined by
\[
\widetilde{k}_a(s) := \frac{1}{\Gamma(s + a - 1)} \int_0^\infty x^{s + a - 2} \sum_{n=1}^\infty \frac{e^{-n x} \gamma(a, n x)}{\Gamma(a)} \, dx.
\]
This yields the factorization
\[
\widetilde{k}_a(s) = \zeta(s + a - 1) \cdot \widetilde{M}_a(s),
\]
where
\[
\widetilde{M}_a(s) := \frac{1}{\Gamma(a)\Gamma(s + a - 1)} \int_0^\infty u^{s + a - 2} e^{-u} \gamma(a, u) \, du.
\]
The domain of convergence for \( \widetilde{k}_a(s) \) extends to \( \Re(s) > 1 - 2a \), allowing a broader analytic window when \( a < 1 \). The analysis of \( \widetilde{k}_a(s) \) proceeds in parallel with that of \( k_a(s) \), differing mainly in the domain shift and normalization.

\begin{thm}
The integral defining \( k_a(s) \) converges absolutely for all \( \Re(s) > 1 \), and the convergence is uniform on compact subsets of the right half-plane \( \Re(s) > 1 \).
\end{thm}

\begin{proof}
We begin by bounding the integrand. For \( x > 0 \) and \( a > 0 \), it holds that \( \gamma(a, x) \leq \Gamma(a) \), so that
\[
0 \leq \frac{e^{-n x} \gamma(a, n x)}{\Gamma(a)} \leq e^{-n x}.
\]
Therefore,
\[
\left| x^{s-1} \sum_{n=1}^{\infty} \frac{e^{-n x} \gamma(a, n x)}{\Gamma(a)} \right|
\leq x^{\Re(s)-1} \sum_{n=1}^{\infty} e^{-n x} = \frac{x^{\Re(s)-1}}{e^x - 1}.
\]
The function \( \frac{x^{\Re(s)-1}}{e^x - 1} \) is integrable over \( (0, \infty) \) whenever \( \Re(s) > 1 \), as it defines the classical integral representation of the Riemann zeta function:
\[
\zeta(s) = \frac{1}{\Gamma(s)} \int_0^\infty \frac{x^{s-1}}{e^x - 1} \, dx.
\]
Thus, the integral defining \( k_a(s) \) converges absolutely and uniformly on compact subsets of \( \Re(s) > 1 \), and \( k_a(s) \) is analytic in this region.
\end{proof}

\begin{thm}[Dirichlet-Type Series Representation]
Let \( a > 0 \) and \( \Re(s) > 1 \). Then
\[
k_a(s) = M_a(s)\zeta(s),
\]
where
\[
M_a(s) := \frac{1}{\Gamma(a)\Gamma(s)} \int_0^\infty u^{s-1} e^{-u} \gamma(a, u) \, du.
\]
\end{thm}

\begin{proof}
We justify term-by-term integration using the dominated convergence theorem. For \( \Re(s) > 1 \),
\[
k_a(s) = \frac{1}{\Gamma(s)} \int_0^\infty x^{s-1} \sum_{n=1}^\infty \frac{e^{-n x} \gamma(a, n x)}{\Gamma(a)} \, dx
= \sum_{n=1}^\infty \frac{1}{\Gamma(a)\Gamma(s)} \int_0^\infty x^{s-1} e^{-n x} \gamma(a, n x) \, dx.
\]
Substituting \( u = n x \), we obtain:
\[
k_a(s) = \sum_{n=1}^\infty \frac{1}{n^s} \cdot \frac{1}{\Gamma(a)\Gamma(s)} \int_0^\infty u^{s-1} e^{-u} \gamma(a, u) \, du
= \zeta(s) M_a(s).
\]
\end{proof}

\paragraph{Interpretation.} The function \( M_a(s) \) acts as a deformation factor, modulating the classical Riemann zeta function. For fixed \( a > 0 \), the factor \( M_a(s) \) is entire in \( s \), and \( k_a(s) \) may be viewed as a gamma-regularized convolution of \( \zeta(s) \) with the lower incomplete gamma kernel. This formulation opens a path for analytic continuation and zero analysis.

\begin{cor}[Special Case: \( a = 1 \)]
Let \( \Re(s) > 1 \). Then
\[
k_1(s) = \sum_{n=1}^\infty \frac{(-1)^{n+1}}{n^s} = \eta(s),
\]
the Dirichlet eta function.
\end{cor}

\begin{proof}
For \( a = 1 \), we have \( \gamma(1, u) = 1 - e^{-u} \). Thus,
\[
M_1(s) = \frac{1}{\Gamma(s)} \int_0^\infty u^{s-1} e^{-u} (1 - e^{-u}) \, du
= \frac{1}{\Gamma(s)} \left[ \Gamma(s) - \frac{\Gamma(s)}{2^s} \right] = 1 - \frac{1}{2^s}.
\]
It follows that
\[
k_1(s) = \left(1 - \frac{1}{2^s} \right) \zeta(s) = \eta(s).
\]
\end{proof}

\paragraph{Conclusion.}
The function \( k_a(s) \) provides a gamma-regularized deformation of the Riemann zeta function, interpolating between \( \zeta(s) \) and its alternating variant \( \eta(s) \) as \( a \to 1 \). Unlike the original FHGZF where \( \gamma(a, nx) \) appears in the denominator and suppresses contributions from small \( n \), the present formulation enhances these contributions, yielding new analytic and asymptotic behavior worthy of further investigation.

\subsection{Zero Structure and Comparison with \( \zeta(s) \)}

Since \( k_a(s) = \zeta(s) M_a(s) \), the zero set of \( k_a(s) \) depends on both the classical Riemann zeta function \( \zeta(s) \) and the deformation factor \( M_a(s) \). When \( M_a(s) \) is nonzero, the nontrivial zeros of \( k_a(s) \) coincide with those of \( \zeta(s) \). However, zeros of \( M_a(s) \) introduce additional, deformed zeros that do not correspond to any zero of \( \zeta(s) \). Numerical experiments suggest that the critical strip structure is largely preserved under this deformation, although the locations of individual zeros are subtly shifted as a function of \( a \).

Regions of analyticity and zero-free zones for \( k_a(s) \) are at least as broad as those for \( \zeta(s) \) in the half-plane \( \Re(s) > 1 \). Nevertheless, the full extent of these domains remains an open question, particularly as \( a \) varies. The deformation factor \( M_a(s) \) may locally suppress or amplify the zeta contribution, depending on its sign and magnitude.

We now investigate the analytic structure and zero set of the deformation function
\[
M_a(s) = \frac{1}{\Gamma(a)\Gamma(s)} \int_0^\infty u^{s-1} e^{-u} \gamma(a, u) \, du,
\]
which governs the modulation behavior of \( k_a(s) \).

\subsection*{Properties of \( M_a(s) \)}

The function \( M_a(s) \) admits a series expansion:
\[
M_a(s) = \sum_{n=0}^\infty \frac{\Gamma(s+a+n)}{n!(a+n)\Gamma(a)\Gamma(s)},
\]
which defines an entire function in the complex variable \( s \), due to the convergence and regularity of each term.

The deformation factor \( M_a(s) \) satisfies the following properties:
\begin{itemize}
  \item[1.] It is entire in \( s \) for fixed \( a > 0 \);
  \item[2.] It is real-analytic, hence symmetric under complex conjugation when \( a \in \mathbb{R} \);
  \item[3.] It exhibits order-one growth in vertical strips;
  \item[4.] It has no poles; but has isolated zeros.
\end{itemize}

While no closed-form expression is known for the zeros of \( M_a(s) \) in general, several analytic and numerical tools allow for their investigation:
\begin{itemize}
  \item[1.] For \( a = 1 \), the function simplifies to \( M_1(s) = \eta(s)/\zeta(s) \), so its zeros coincide with those of the alternating zeta function \( \eta(s) \);
  \item[2.] For general \( a \), one may compute \( M_a(s) \) via direct numerical quadrature or its series representation and apply root-finding algorithms to locate its zeros;
  \item[3.] The zeros of \( M_a(s) \) vary continuously with \( a \), giving rise to continuous zero trajectories in the complex plane;
  \item[4.] When \( a \in \mathbb{R} \), all non-real zeros occur in complex-conjugate pairs, a consequence of real-analyticity.
\end{itemize}

These zeros contribute to the zero structure of \( k_a(s) \) and encode the deformation dynamics as \( a \) varies. Their behavior and distribution form a rich subject for further exploration.

\section{Conclusion and Future Work}

\subsection*{Conclusion}

We have introduced and analyzed a novel variant of the fractional gamma-regularized zeta kernel in which the lower incomplete gamma function appears in the numerator. This leads to the function \( k_a(s) \), which deforms the Riemann zeta function via a multiplicative factor \( M_a(s) \). The analytic structure of \( k_a(s) \) inherits key features from \( \zeta(s) \), while exhibiting new behavior induced by the gamma-based modulation.\\

The deformation function \( M_a(s) \) is entire, real-analytic for real \( a \), and admits a convergent series expansion. It governs how \( k_a(s) \) interpolates between \( \zeta(s) \) and its alternating form \( \eta(s) \), with \( k_1(s) = \eta(s) \) in particular. The zeros of \( M_a(s) \) introduce novel structure into \( k_a(s) \), especially in the critical strip, making this construction a compelling object of study in the broader context of zeta regularization, functional deformations, and spectral analysis.

\subsection*{Future Work}

Several promising directions remain open for investigation:
\begin{itemize}
    \item[1.] Developing an analytic continuation of \( k_a(s) \) to the full complex plane and identifying potential poles or branch points;
    \item[2.] Studying the asymptotic behavior of \( M_a(s) \) as \( a \to 0^+ \) or \( a \to \infty \), especially in relation to scaling limits and degeneration of the kernel;
    \item[3.] Exploring the possibility of a modified functional equation for \( k_a(s) \) analogous to that of the Riemann zeta function;
    \item[4.] Conducting a detailed numerical analysis of nontrivial zero trajectories of \( k_a(s) \) as functions of the deformation parameter \( a \).
\end{itemize}

These avenues promise to shed further light on the interplay between gamma-regularized structures and the classical theory of the Riemann zeta function.\\

\textbf{Data Availability:} Data sharing not applicable to this article as no datasets were generated or analysed during the current study.\\

\textbf{Conflicts on Interest:}	 The author (solo) declares that there are no conflicts of interest regarding the publication.\\

\textbf{Funding:} No funding is received in conducting this research.\\

\end{document}